\documentclass[12pt, reqno]{amsart}
\usepackage{amsmath, amsthm, amscd, amsfonts, amssymb, graphicx, color}
\usepackage[bookmarksnumbered, colorlinks, plainpages]{hyperref}
\hypersetup{colorlinks=true,linkcolor=red, anchorcolor=green, citecolor=cyan, urlcolor=red, filecolor=magenta, pdftoolbar=true}

\newtheorem{theorem}{Theorem}[section]
\newtheorem{lemma}[theorem]{Lemma}

\newtheorem{cor}[theorem]{Corollary}
\theoremstyle{definition}

\theoremstyle{remark}
\newtheorem{remark}[theorem]{\bf{Remark}}
\numberwithin{equation}{section}
\begin{document}
		\title[Enhancement of the Cauchy-Schwarz Inequality and  Its Implications ... ]  {Enhancement of the Cauchy-Schwarz Inequality and Its Implications for Numerical Radius Inequalities}
	\author[R.K. Nayak]{  Raj Kumar Nayak}

	\address{(Nayak) Department of Mathematics, GKCIET, Malda, India}
\email{rajkumar@gkciet.ac.in/rajkumarju51@gmail.com}
	\subjclass[2010]{Primary 47A12, 47A30, Secondary  47A63, 15A60}
	\keywords{Cauchy-Schwarz inequality; Numerical radius; Norm inequality}
	\date{}
	\maketitle 
	\begin{abstract}
		In this article, we establish an improvement of the Cauchy-Schwarz inequality. Let $x, y \in \mathcal{H},$ and let $f: (0,1) \rightarrow \mathbb{R}^+$ be a well-defined function, where $\mathbb{R}^+$ denote the set of all positive real numbers. Then \[|\langle x, y \rangle|^2 \leq \frac{f(t)}{1+f(t)} \|x\|^2 \|y\|^2  + \frac{1}{1+ f(t)} |\langle x, y \rangle | \|x\|\|y\|. \] We have applied this result to derive new and improved upper bounds for the numerical radius. 
		\end{abstract}
		\section{Introduction}
	Throughout the article, $\mathcal{H}$ represents the complex Hilbert space with the inner product $\langle \cdot, \cdot \rangle$. The Cauchy-Schwarz inequality is a fundamental result in linear algebra and analysis, which says that for any vectors $x, y \in \mathcal{H},$ \[|\langle x, y \rangle| \leq \|x\| \|y\|. \]
	This inequality establishes a crucial relationship between the geometric and algebraic properties of vectors. The Cauchy-Schwarz inequality is instrumental in various fields such as functional analysis, probability theory, and quantum mechanics, where it facilitates the derivation of numerous theoretical results and practical applications. Therefore, enhancing this fundamental inequality holds significant importance for researchers. In recent times, several researchers have improved the Cauchy-Schwarz inequality through various approaches. We encourage readers to explore the referenced articles \cite{Kittaneh MIA,  filomat, Alomari Ricerche,  Nayak IIST, Stojiljković} for further insight.  
	\noindent	Let $\mathcal{B}({\mathcal{H}})$ denote the $C^*$-algebra of all bounded linear operators on a complex Hilbert space $\mathcal{H}.$  The absolute value of $T,$ is defined as $|T| = (T^*T)^{\frac{1}{2}},$ where $T^*$ represents the Hilbert adjoint of the operator $T.$ The numerical range of $T \in \mathcal{B}(\mathcal{H})$ is denoted by $W(T)$, is the image of the unit sphere of $\mathcal{H}$ under the mapping $x \rightarrowtail \langle Tx, x \rangle.$ The numerical radius and operator norm of an operator  $T \in \mathcal{B}(\mathcal{H})$, are denoted by $w(T)$ and $\|T\|$ respectively, and are defined as 
		
		 \[w(T) = \sup_{\|x\|=1} \left|\langle Tx, x \rangle \right|\]
		and \[ \|T\| = \sup_{\|x\|=1} \|Tx\|.\]

	It is widely recognized that $w(\cdot)$ constitutes a norm on 
$\mathcal{B}(\mathcal{H})$  which is comparable to the standard operator norm through the following inequality:	
		\begin{equation}\label{eq 1}
		 \frac{\|T\|}{2} \leq w(T) \leq \|T\|.
		 \end{equation}
		 Several advancements of this inequality have been made in recent years.   Here, we enumerate a few of these improvements:\\ 
		 \noindent In \cite{kstudiamath1}, Kittaneh established that if $T \in \mathcal{B}(\mathcal{H})$, then 
		  \begin{equation} \label{eq kittaneh}
		  w(T) \leq \frac{1}{2} \left\||T| + |T^*|\right\|. 
		  \end{equation}
		  This inequality was subsequently extended by El-Haddad et al. \cite{kstudiamath3}, asserting that, if $T \in \mathcal{B}(\mathcal{H}),$ then 
		   \begin{equation}\label{eq el kit}
		   w^{2r}(T) \leq \frac{1}{2} \left\||T|^{2r} + |T^*|^{2r}\right\|.
		   \end{equation}
		   		  In 2015,  Abu-Omar et al. \cite{abu omar} proved that for $T \in \mathcal{B}(\mathcal{H}),$
		   		   \begin{equation}\label{eq 3}
		   		   w^2(T) \leq \frac{1}{4} \left\||T|^2 + |T^*| \right\| + \frac{1}{2} w(T^2).
		   		   \end{equation}
		   		   In 2021, Bhunia et al. \cite{bulletin sci} proved that, for $T \in \mathcal{B}(\mathcal{H}),$
		   		   \begin{equation}\label{eq 11}
		   		   w^2(T) \leq \frac{1}{4} \left\||T|^2 + |T^*|^2 \right\| + \frac{1}{2} w\left(|T||T^*| \right).
		   		   \end{equation} 
		   		   
		  Subsequently, Dragomir \cite{dragomir} established the following inequality for the product of two operators, asserting that for $T, S \in \mathcal{B}(\mathcal{H})$ and $r \geq 1$,
		   \begin{equation} \label{eq dragomir}
		   w^r(S^*T) \leq \frac{1}{2} \left\||T|^{2r} + |S|^{2r} \right\|.
		   \end{equation}
		   Recently, Al-Dolat et al. \cite{filomat} improved inequality (\ref{eq dragomir}), stating that for $T, S \in \mathcal{B}(\mathcal{H})$ and $\lambda \geq 0,$ 
		   \begin{equation} \label{eq 4}
		   w^2(S^*T) \leq \frac{1}{2(1+\lambda)} \left\||T|^2 + |S|^2 \right\| w(S^*T) + \frac{\lambda}{2(1+\lambda)} \left\||T|^4 + |S|^4 \right\|.
		   \end{equation}
		   For further details on recent work regarding the numerical radius inequalities, readers are referred to \cite{dragomir lama, Pintu laa 2024, book, raj, Nayak Acta, Sahoo}.
		   The renowned Young inequality asserts that for any two positive real numbers $x$ and $y$, along with $t$ lying in the interval $[0,1]$, the following relation holds true: 
		   \begin{equation} \label{eq 7}
		   x^ty^{1-t} \leq tx + (1-t) y.
		   \end{equation} 
		 For $t = \frac{1}{2}$, we encounter the famous arithmetic-geometric mean inequality, which stipulates that for any two positive real numbers $x$ and $y$, the following inequality holds true:
		   \begin{equation}\label{eq am-gm}
		   \sqrt{xy} \leq \frac{x+y}{2}.
		   \end{equation}
		In this article, we present two advancements of the Cauchy-Schwarz inequality as follows:\\
		 	Let $f: (0,1) \rightarrow \mathbb{R}^+$ be a well-defined function. Then, for any $x, y \in \mathcal{H},$
		 \begin{equation} \label{gencauchy}
	 |\langle x, y \rangle|^2 \leq \frac{f(t)}{1+f(t)} \|x\|^2 \|y\|^2  + \frac{1}{1+ f(t)} |\langle x, y \rangle | \|x\|\|y\| \leq \|x\|^2 \|y\|^2 
		 \end{equation}
		 and 
		 \begin{equation} \label{impC-S 2}
		 	|\langle x, y \rangle|^2 \leq \frac{f(t)}{2(1+ f(t))} \|x\|^2 \|y|^2 + \frac{2+ f(t)}{2(1+ f(t))} |\langle x, y \rangle| \|x\| \|y\|.
		 \end{equation}
		We have utilized this to derive several upper bounds for numerical radius inequalities for operators and products of operators, enhancing and refining inequalities (\ref{eq 1}),  (\ref{eq kittaneh}),  (\ref{eq el kit}), (\ref{eq 11}), (\ref{eq dragomir}), (\ref{eq 4}).
		\section{Main Results}
	In this section, we will establish various improvements of the upper bounds for numerical radii. To achieve this, we will utilize the following well-known lemmas. The first lemma is a consequence of the spectral theorem in conjunction with  Jensen's inequality.
		\begin{lemma}\label{positive op}\cite{mccarthy}
			Let $T \in \mathcal{B}(\mathcal{H})$ be a positive operator and $x$ be an unit vector in $\mathcal{H}.$ Then, for $r \geq 1,$ the inequality \[\langle Tx, x \rangle^r \leq \langle T^rx, x \rangle \]
			holds.
		\end{lemma}
	The next lemma is a norm inequality for a non-negative convex function.
	\begin{lemma}\label{convex op} \cite{aujla}
		Let $h$ be a non-negative convex function on $[0, \infty)$ and $A, B \in \mathcal{B}(\mathcal{H})$ be positive operators. Then  \[\left\|h\left(\frac{A+B}{2}\right)\right\| \leq \left\|\frac{h(A) + h(B)}{2}\right\| \]
		holds.
		In particular if $r \geq 1$, then \[\left\|\left(\frac{A+B}{2} \right)^r \right\| \leq \left\|\frac{A^r+B^r}{2}\right\|. \]
	\end{lemma}
The subsequent result presents a generalized form of the mixed Schwarz inequality.
	\begin{lemma}\label{mixed schwarz} \cite{Res}
		Let $T \in \mathcal{B}(\mathcal{H}),$ and $x, y \in \mathcal{H}$. Let $g, h$ are two non-negative continuous function on $[0, \infty) $ satisfying $g(t) h(t)=t.$ Then \[|\langle Tx, y \rangle| \leq \left\| g(|T|)x\right\| \left\|h(|T^*|)y \right\|.\]
	\end{lemma}
The next result is the famous Buzano's generalization of the Cauchy-Schwarz inequality.
\begin{lemma}\label{buzano}\cite{Buzano}
	Let $x, y, e \in \mathcal{H}$ with $\|e\|=1.$ Then  \[|\langle x, e\rangle \langle e , y\rangle| \leq \frac{1}{2}\left(\|x\|\|y\| + |\langle x, y \rangle| \right) .\]
\end{lemma}
Recently, Dolat et al. \cite{filomat} presented a new improvement of the Cauchy-Schwarz inequality. The following lemma states:
\begin{lemma} \cite{filomat} \label{cauchyimp}
	Let $x, y \in \mathcal{H}.$ Then for any $\lambda \geq 0,$  \[|\langle x, y \rangle|^2 \leq \frac{1}{\lambda+1} \|x\|\|y\||\langle x, y\rangle| + \frac{\lambda}{\lambda + 1}\|x\|^2\|y\|^2 \leq \|x\|^2\|y\|^2. \]
\end{lemma}
The next lemma is the operator version of the classical Jensen's inequality.
\begin{lemma}\label{jenson} \cite{Res}
	Let $T \in \mathcal{B}(\mathcal{H})$ be a self-adjoint operator whose spectrum contained in the interval $J$, and let $x \in \mathcal{H}$ be a unit vector. If $h$ is a convex function on $J$, then \[h(\langle Tx, x \rangle) \leq \langle h(T)x, x \rangle. \]
\end{lemma}

Now, we will demonstrate the following improvement of the Cauchy-Schwarz inequality, which will be utilized frequently. 

\begin{lemma}\label{gen cauchy}
Let $f: (0,1) \rightarrow \mathbb{R}^+$ be a well-defined function. Then, for any $x, y \in \mathcal{H},$ \[|\langle x, y \rangle|^2 \leq \frac{f(t)}{1+f(t)} \|x\|^2 \|y\|^2  + \frac{1}{1+ f(t)} |\langle x, y \rangle | \|x\|\|y\| \leq \|x\|^2 \|y\|^2 .\] 
\end{lemma}
\begin{proof}
Employing the Cauchy-Schwarz inequality, we obtain 
 \begin{eqnarray*}
 	|\langle x, y \rangle|^2 &\leq& 	|\langle x, y \rangle| \|x\| \|y\|\\
 	&\leq& 	|\langle x, y \rangle| \|x\| \|y\| + f(t) \left(\|x\|^2 \|y\|^2 -	|\langle x, y \rangle|^2 \right)\\
 	\Rightarrow 	|\langle x, y \rangle|^2 &\leq& \frac{f(t)}{1+f(t)} \|x\|^2 \|y\|^2  + \frac{1}{1+ f(t)} |\langle x, y \rangle | \|x\|\|y\|.
 \end{eqnarray*}  
This establishes the first inequality, and the second inequality follows directly from the Cauchy-Schwarz inequality.
\end{proof}

By setting $f(t)=t$ and $t = \lambda$ in Lemma \ref{gen cauchy}, we derive the following corollary.
\begin{cor}
	Let $x, y \in \mathcal{H}$ and $\lambda \in (0,1),$ then 
	\[|\langle x, y \rangle|^2 \leq \frac{1}{\lambda+1} \|x\|\|y\||\langle x, y\rangle| + \frac{\lambda}{\lambda + 1}\|x\|^2\|y\|^2. \]
\end{cor}
By setting $f(t) = \frac{t}{1-t}$ in the Lemma \ref{gen cauchy}, we obtain the following corollary. 
\begin{cor} \cite {Alomari Ricerche}
	Let $x, y \in \mathcal{H},$ and $\theta \in (0,1),$ then,
	\[|\langle x, y \rangle|^2  \leq t\|x\|^2 \|y\|^2 + (1-t) |\langle x, y \rangle| \|x\|\|y\| .\]
\end{cor}
Our first theorem provides an upper bound for the numerical radius of the product of two operators.
\begin{theorem}\label{th2}
	Let $T, S \in \mathcal{B}(\mathcal{H}),$ and $f: (0,1) \rightarrow \mathbb{R}^+$ be a well-defined function. Then for $r \geq 1,$  
	\begin{eqnarray*}
	w^{2r}(T^*S) &\leq& \frac{1}{2(1+ f(t))} w^r (T^*S)  \left\||T|^{2r} + |S|^{2r}\right\| +  \frac{f(t)}{4(1+ f(t))} \left\||T|^{4r} + |S|^{4r}\right\| \\ && + \frac{f(t)}{2(1+ f(t))}  w\left( |S|^{2r} |T|^{2r}\right).
	\end{eqnarray*}
\end{theorem}
\begin{proof}
	Let $x \in \mathcal{H}$ with $\|x\|=1.$ Then we have 
\begin{eqnarray*}
	\left|\langle T^*Sx, x \rangle \right|^{2r} &=& \left|\langle Tx, Sx \rangle \right|^{2r}\\
	&\leq& \left( \frac{1}{(1+ f(t))}  \|Tx\|\|Sx\| |\langle Tx, Sx \rangle| +  \frac{f(t)}{(1+ f(t))} \|Tx\|^2\|Sx\|^2 \right)^r\\&&\,\,\,\,\,\,\,\,\,\,\,\,\,\,\,\,\,\,\,\,\,\,\,\,\,\,\,\,\,\,\,\,\,\,\,\,\,\,\,\,\,\,\,\,\,\,\,\,\,\,\,\,\,\,\,\,\,\,\,\,\,\,\,\,\,\,\,\,\,\,\,\,\,\,\,\,\,\,\,\,\,\,\,\,\,\,\,\,\,\,\,\,\,\,\,\,\,\,\,\,\,\,\,\,\,\, (\mbox{using Lemma \ref{gen cauchy}})\\
	&\leq& \frac{1}{(1+ f(t))}  \|Tx\|^{r} \|Sx\|^r |\langle T^*Sx, x \rangle|^r + \frac{f(t)}{(1+ f(t))} \|Tx\|^{2r} \|Sx\|^{2r}\\
	&&\,\,\,\,\,\,\,\,\,\,\,\,\,\,\,\,\,\,\,\,\,\,\,\,\,\,\,\,\,\,\,\,\,\,\,\,\,\,\,\,\,\,\,\,\,\,\,\,~(\mbox{using the convexity of the function $f(t) = t^r$ })\\
	&\leq& \frac{1}{2(1+ f(t))} \left\langle \left(|T|^{2r}+ |S|^{2r} \right)x, x \right\rangle \left|\langle T^*Sx, x \rangle \right|^r\\&& + \frac{f(t)}{(1+ f(t))} \langle |T|^{2r}x, x \rangle \langle x, |S|^{2r}x \rangle\\ 	&&\,\,\,\,\,\,\,\,\,\,\,\,\,\,\,\,\,\,\,\,\,\,\,\,\,\,\,\,\,\,\,\,\,\,\,\,\,\,\,\,\,\,\,\,\,\,\, ~(\mbox{using A.M-G.M inequality and Lemma \ref{positive op}})\\ 
	&\leq& \frac{1}{2(1+ f(t))}  \left\langle \left(|T|^{2r}+ |S|^{2r} \right)x, x \right\rangle \left|\langle T^*Sx, x \rangle \right|^r \\&&+ \frac{f(t)}{2(1+ f(t))} \left(\left\||T|^{2r}x\right\|\left\||S|^{2r}x \right\| + \left\langle |T|^{2r}x, |S|^{2r}x \right\rangle  \right)~(\mbox{using Lemma \ref{buzano}})\\
	&\leq& \frac{1}{2(1+ f(t))}  \left\langle \left(|T|^{2r}+ |S|^{2r} \right)x, x \right\rangle \left|\langle T^*Sx, x \rangle \right|^r \\&&+ \frac{f(t)}{4(1+ f(t))}  \left\langle \left(|T|^{4r} + |S|^{4r}  \right)x, x \right\rangle + \frac{f(t)}{2(f(t)+1)} \left\langle \left(|S|^{2r} |T|^{2r}\right)x, x \right\rangle\\
	&\leq& \frac{1}{2(1+ f(t))} w^r (T^*S)  \left\||T|^{2r} + |S|^{2r}\right\| +  \frac{f(t)}{4(1+ f(t))} \left\||T|^{4r} + |S|^{4r}\right\| \\ && + \frac{f(t)}{2(1+ f(t))}  w\left( |S|^{2r} |T|^{2r}\right).
\end{eqnarray*}
Now, by taking supremum over $x$ with $\|x\|=1,$  we obtain our desired inequality.
\end{proof}
Now, we present several corollaries that can be derived from Theorem \ref{th2}.
Our first corollary is an improvement of inequality (\ref{eq dragomir}).
\begin{cor}
	Let $T, S \in \mathcal{B}(\mathcal{H}),$ and $f: (0,1) \rightarrow \mathbb{R}^+$ be a well-defined function.  Then for any  $r \geq 1,$
	\begin{eqnarray*}
		w^{2r}(T^*S) &\leq& \frac{1}{2(f(t)+1)}\left\||T|^{2r} + |S|^{2r}\right\|w^r(T^*S) + \frac{f(t)}{4(f(t)+1)} \left\||T|^{4r} + |S|^{4r}\right\|\\&&+ \frac{f(t)}{2(f(t)+1)}w\left(|S|^{2r}|T|^{2r} \right) \\
		&\leq& \frac{1}{2} \left\||T|^{4r} + |S|^{4r} \right\|.
	\end{eqnarray*}
\end{cor}
\begin{proof}
	Using inequality (\ref{eq dragomir}), we obtain
	\begin{eqnarray*}
		w^{2r}(T^*S) &\leq& \frac{1}{2(f(t)+1)}\left\||T|^{2r} + |S|^{2r}\right\|w^r(T^*S) + \frac{f(t)}{4(f(t)+1)} \left\||T|^{4r} + |S|^{4r}\right\|\\&&+ \frac{f(t)}{2(f(t)+1)}w\left(|S|^{2r}|T|^{2r} \right) \\
		&\leq& 	 \frac{1}{4(f(t)+1)}\left\||T|^{2r} + |S|^{2r}\right\|^2 + \frac{f(t)}{4(f(t)+1)} \left\||T|^{4r} + |S|^{4r}\right\|\\&&+ \frac{f(t)}{4(f(t)+1)} \left\||T|^{4r} + |S|^{4r} \right\|~(\mbox{using inequality (\ref{eq dragomir})})\\
		&\leq&  \frac{1}{2(f(t)+1)}\left\||T|^{4r} + |S|^{4r}\right\| + \frac{f(t)}{4(f(t)+1)} \left\||T|^{4r} + |S|^{4r}\right\|\\&&+ \frac{f(t)}{4(f(t)+1)} \left\||T|^{4r} + |S|^{4r} \right\|~(\mbox{using Lemma \ref{convex op}})\\
		&=& \frac{1}{2} \left\| |T|^{4r} + |S|^{4r}\right\|.
	\end{eqnarray*}
\end{proof}	
\begin{cor} \cite{Alomari Ricerche}
	Let $T, S \in \mathcal{B}(\mathcal{H}).$ Then for any $t \in (0,1)$ and $r \geq 1,$ \[ w^{2r}(S^*T) \leq \frac{1-t}{2}\left\||T|^{2r} + |S|^{2r}\right\|w^r(T^*S) + \frac{t}{2} \left\||T|^{4r} + |S|^{4r}\right\|.\]
\end{cor}
\begin{proof}
By setting $f(t) = \frac{t}{1-t}$ and utilizing inequality (\ref{eq dragomir}), we derive the previously stated inequality. Therefore, it's evident that our theorem in Theorem \ref{th2} generalizes and improves \cite[Th. 1]{Alomari Ricerche} for $t \in (0,1)$.
\end{proof}
\begin{cor} \cite{Kittaneh MIA}
Let 	$T, S \in \mathcal{B}(\mathcal{H}).$ Then  \[ w^{2}(S^*T) \leq \frac{1}{3}\left\||T|^{2} + |S|^{2}\right\|w(T^*S) + \frac{1}{6} \left\||T|^{4} + |S|^{4}\right\|.\]
\end{cor}
\begin{proof}
By setting $f(t) = \frac{1}{2}$ in Theorem \ref{th2} and utilizing inequalities (\ref{eq dragomir}) and Lemma (\ref{convex op}), we establish the aforementioned inequality, previously proven by Kittaneh and Moradi \cite[Th. 1]{Kittaneh MIA}. Therefore our inequality in Theorem \ref{th2} generalizes and improves the bound established in \cite[Th. 1]{Kittaneh MIA}.
\end{proof}
\begin{cor}\label{aldolat filomat}
	Let $T, S \in \mathcal{B}(\mathcal{H})$, and $t \in (0,1)$, then 
	\begin{eqnarray*}
		w^2(T^*S) &\leq& \frac{1}{2(t+1)}\left\||T|^2 + |S|^2 \right\| w(T^*S) + \frac{t}{4(t+1)} \left\| |T|^{4} + |S|^4 \right\|\\ && + \frac{t}{2(t+1)} w\left( |S|^2|T|^2\right)\\
		&\leq& \frac{1}{2(t+1)} \left\||T|^2 + |S|^2 \right\| w(T^*S) + \frac{t}{2(t+1)} \left\||T|^4 + |S|^4 \right\|.
	\end{eqnarray*}
\end{cor}
\begin{proof}
Taking $r=1$ and $f(t) =t$ in Theorem \ref{th2} yields the first inequality, with the second following from inequality (\ref{eq dragomir}). Thus, our inequality in Theorem \ref{th2} extends and enhances the inequality derived by Al-Dolat et al. \cite[Th. 2.6]{filomat} for $t \in (0,1)$.
\end{proof}

If we choose $f(t)=t$ in Theorem \ref{th2}, we derive the subsequent corollary, previously proved by Nayak \cite[Th. 2.16]{Nayak IIST} for $t \in (0,1)$.
\begin{cor} 
	Let $T, S \in \mathcal{B}(\mathcal{H}).$ Then for $t \in (0,1)$, and $r \geq 1,$
	\begin{eqnarray*}
		w^{2r}(T^*S) &\leq& \frac{1}{2(t+1)}\left\||T|^{2r} + |S|^{2r }\right\| w^r(T^*S) + \frac{t}{4(t+1)} \left\| |T|^{4r} + |S|^{4r} \right\|\\ && + \frac{t}{2(t+1)} w\left( |S|^{2r}|T|^{2r}\right).
	\end{eqnarray*}
\end{cor}
\begin{theorem} \label{th3}
		Let $T \in \mathcal{B}(\mathcal{H}),$ and $g, h$ are two non-negative continuous function on $[0, \infty) $ satisfying $g(t) h(t)=t.$ Then,  for a well-defined function $f:(0,1) \rightarrow \mathbb{R}^+$ \begin{eqnarray*}
			w^2(T) &\leq& \frac{f(t)}{4(1+ f(t))} \left\|g^4(|T|) + h^4(|T^*|)\right\| + \frac{f(t)}{2(1+ f(t))}w\left( h^2(|T^*|) g^2(|T|)\right) \\&&+ \frac{1}{2(1+ f(t))} w(T) \left\|g^2(|T|) + h^2(|T^*|) \right\|.
		\end{eqnarray*}
\end{theorem}
\begin{proof}
	Let $x \in \mathcal{H}.$ Then we have 
	\begin{eqnarray*}
		|\langle Tx, x \rangle|^2 &=& \frac{f(t)}{(1+ f(t))} |\langle Tx, x \rangle|^2  + \frac{1}{1+ f(t)} |\langle Tx, x \rangle|^2 \\
		&\leq& \frac{f(t)}{(1+ f(t))} \left\langle g^2(|T|)x,x \right\rangle   \left\langle h^2(|T|)x,x \right\rangle\\ && + \frac{1}{1+f(t)} |\langle Tx, x \rangle| \sqrt{\left\langle g^2(|T|)x,x \right\rangle   \left\langle h^2(|T|)x,x \right\rangle}~~(\mbox{using Lemma \ref{mixed schwarz}})\\
		&\leq& \frac{f(t)}{2(1+ f(t))} \left( \left\|g^2(|T|)x\right\|  \left\|h^2(|T|)x\right\| + \left\langle g^2(|T|)x, h^2(|T^*|)x \right\rangle  \right)\\&& + \frac{1}{1+f(t)} |\langle Tx, x \rangle| \sqrt{\left\langle g^2(|T|)x,x \right\rangle   \left\langle h^2(|T|)x,x \right\rangle}~~(\mbox{using Lemma \ref{buzano}})\\
		&\leq& \frac{f(t)}{4(1+ f(t))} \left\langle \left\{g^4(|T|) + h^4(|T^*|) \right\}x, x \right\rangle + \frac{f(t)}{2(1+ f(t))} \left\langle h^2(|T^*|) g^2(|T|)x, x \right\rangle \\ &&+ \frac{1}{2(1+f(t))} |\langle Tx, x \rangle|   \left\langle \left\{g^2(|T|) + h^2(|T^*|) \right\}x, x \right\rangle\\
	&&	\,\,\,\,\,\,\,\,\,\,\,\,\,\,\,\,\,\,\,\,\,\,\,\,\,\,\,\,\,\,\,\,\,\,\,\,\,\,\,\,\,\,\,\,\,\,\,\,\,\,\,\,\,\,\,\,\,\,\,\,\,\,\,\,\,\,\,\,\,\,\,\,\,\,\,\,\,\,\,\,\,\,\,\,\,\,\,\,\,\,\,\,\,\,\,\,\,\,\,\, (\mbox{using the A.M-G.M inequality})\\
		&\leq& \frac{f(t)}{4(1+ f(t))} \left\|g^4(|T|) + h^4(|T^*|)\right\| + \frac{f(t)}{2(1+ f(t))}w\left( h^2(|T^*|) g^2(|T|)\right) \\&&+ \frac{1}{2(1+ f(t))} w(T) \left\|g^2(|T|) + h^2(|T^*|) \right\|.
	\end{eqnarray*}
	Taking the supremum over all $x$ with $\|x\| =1,$ we attain our desired inequality.
\end{proof}
The following corollaries can be derived from Theorem \ref{th3}.

\begin{cor}
	Let $T \in \mathcal{B}(\mathcal{H}),$ and $f: (0,1) \rightarrow \mathbb{R}$ be a well-defined function. Then 
	\begin{eqnarray*}
	w^2(T) &\leq& 	\frac{f(t)}{4(1+ f(t))} \left\||T|^2 + |T^*|^2\right\| + \frac{f(t)}{2(1+ f(t))}w\left( |T^*||T|\right) \\&&+ \frac{1}{2(1+ f(t))} w(T) \left\||T| + |T^*|\right\|\\ &\leq& \frac{1}{2} \left\| |T|^2 + |T^*|^2 \right\|.
	\end{eqnarray*}
\end{cor}
\begin{proof}
	Let us assume $g(t) = h(t) = \sqrt{t}$ in Theorem \ref{th3}. 
Then, by using the inequality (\ref{eq kittaneh}) and (\ref{eq dragomir}), we obtain 
\begin{eqnarray*}
		w^2(T) &\leq& 	\frac{f(t)}{4(1+ f(t))} \left\||T|^2 + |T^*|^2\right\| + \frac{f(t)}{2(1+ f(t))}w\left( |T^*||T|\right) \\&&+ \frac{1}{2(1+ f(t))} w(T) \left\||T| + |T^*|\right\|\\ &\leq& \frac{f(t)}{4(1+ f(t))} \left\||T|^2 + |T^*|^2\right\| +  \frac{f(t)}{4(1+ f(t))} \left\||T|^2 + |T^*|^2\right\|\\ && + \frac{1}{4(1+ f(t))}  \left\||T| + |T^*|\right\|^2\\
		 &\leq&  \frac{f(t)}{4(1+ f(t))} \left\||T|^2 + |T^*|^2\right\| +  \frac{f(t)}{4(1+ f(t))} \left\||T|^2 + |T^*|^2\right\|\\ && + \frac{1}{2(1+ f(t))}  \left\||T|^2 + |T^*|^2\right\|~~(\mbox{using Lemma \ref{convex op}})\\
		 &=& \frac{1}{2}  \left\||T|^2 + |T^*|^2\right\|.
\end{eqnarray*}
Thus, the upper bound established in Theorem \ref{th3} provides a significant improvement over the inequality (\ref{eq el kit}) when $r=1$.
\end{proof}
\begin{cor} \cite{Kittaneh MIA}
	Let $T \in \mathcal{B}(\mathcal{H}),$ then \[w^2(T) \leq \frac{1}{6} \left\||T|^2 + |T^*|^2\right\| + \frac{1}{3} w(T) \left\||T| + |T^*|\right\|. \]
\end{cor}
\begin{proof}
	Taking $g(t) = h(t) = \sqrt{t},$ and $ f(t) = \frac{1}{2}$ in Theorem \ref{th3}, we get 
	\begin{eqnarray*}
		w^2(T) &\leq& \frac{1}{12} \left\||T|^2 + |T^*|^2\right\| + \frac{1}{6} w\left(|T^*||T|\right) + \frac{1}{3} w(T) \left\||T| + |T^*|\right\|\\
		&\leq& \frac{1}{6} \left\||T|^2 + |T^*|^2\right\|+ \frac{1}{3} w(T) \left\||T| + |T^*|\right\|~~(\mbox{using inequality (\ref{eq dragomir})}.)
	\end{eqnarray*}
Thus, the inequality derived in Theorem \ref{th3} generalizes and improves the upper bound previously obtained by Kittaneh et al. \cite[Th. 2]{Kittaneh MIA}.
	\end{proof}
Next, we demonstrate another refinement of the Cauchy-Schwarz inequality.
	
	\begin{lemma} \label{imp C-S 2}
		Let $x, y \in \mathcal{H},$ and $f: (0,1) \rightarrow \mathbb{R}^+$ be a well-defined function. Then 
		\begin{eqnarray*}
			|\langle x, y \rangle|^2 &\leq& \frac{f(t)}{2(1+ f(t))} \|x\|^2 \|y\|^2 + \frac{2+ f(t)}{2(1+ f(t))} |\langle x, y \rangle| \|x\| \|y\|\\
			&\leq& \|x\|^2 \|y\|^2.
		\end{eqnarray*}
	\end{lemma}
	\begin{proof}
		Using Lemma \ref{gen cauchy} we obtain, 
		\begin{eqnarray*}
		2 |\langle x, y \rangle|^2 &=&  |\langle x, y \rangle|^2 +  |\langle x, y \rangle|  |\langle x, y \rangle|\\
		&\leq& \frac{f(t)}{(1+ f(t))} \|x\|^2 \|y\|^2 + \frac{1}{(1+ f(t))} |\langle x, y \rangle| \|x\| \|y\| + \|x\| \|y\| |\langle x, y \rangle |\\
		\Rightarrow |\langle x, y \rangle|^2 &\leq& \frac{f(t)}{2(1+ f(t))} \|x\|^2 \|y|^2 + \frac{2+ f(t)}{2(1+ f(t))} |\langle x, y \rangle| \|x\| \|y\|.\\
		\end{eqnarray*}
		This establishes the first inequality, while the second inequality follows from the Cauchy-Schwarz inequality.
	\end{proof}
	
Utilizing the preceding improvement of the Cauchy-Schwarz inequality, we proceed to prove our next lemma.

	\begin{lemma} \label{imp buzano}
		Let $x, y, e \in \mathcal{H}$ with $\|e\| =1.$ Then for a well-defined function \\ $f: (0,1) \rightarrow \mathbb{R}^+$, 
		\begin{eqnarray*}
			\left|\langle x, e \rangle \langle e, y \rangle\right|^2 \leq \frac{2+ 3 f(t)}{8(1+f(t))} \|x\|^2 \|y\|^2 + \frac{6+ 5 f(t)}{8(1+ f(t))} \|x\| \|y\| |\langle x, y \rangle|.
		\end{eqnarray*}
	\end{lemma}
	\begin{proof}
		Utilizing Lemma \ref{buzano} and Lemma \ref{imp C-S 2}, we obtain
		\begin{eqnarray*}
				\left|\langle x, e \rangle \langle e, y \rangle\right|^2 &\leq& \frac{1}{4} \left(\|x\|\|y\| + |\langle x, y \rangle |\right)^2 \\
				&=& \frac{1}{4} \left(\|x\|^2 \|y\|^2 + |\langle x, y \rangle|^2 + 2 \|x\|\|y\| |\langle x, y \rangle| \right)\\
				&\leq& \frac{1}{4} \left(\|x\|^2 \|y\|^2 + \frac{f(t)}{2(1+f(t))} \|x\|^2 \|y\|^2 + \frac{2+  f(t)}{2(1+f(t))} |\langle x, y \rangle| \|x\| \|y\|\right)\\ &&+ \frac{1}{2} |\langle x, y \rangle| \|x\| \|y\| \\ &=& \frac{2+ 3 f(t)}{8(1+f(t))} \|x\|^2 \|y\|^2 + \frac{6+ 5 f(t)}{8(1+ f(t))} \|x\| \|y\| |\langle x, y \rangle|.
		\end{eqnarray*}
	This concludes our lemma.
	\end{proof}
	
Using the lemma stated above, we proceed to demonstrate our next theorem.
	\begin{theorem} \label{th 4}
		Let $T \in \mathcal{B}(\mathcal{H})$ and $f: (0, 1) \rightarrow \mathbb{R}^+$ be a well-defined function. Then 
		\begin{eqnarray*}
			w^4(T) &\leq& \frac{2 + 3 f(t)}{32(1+ f(t))} \left\| |T|^4 + |T^*|^4 \right\| + \frac{2 + 3 f(t)}{16 (1 + f(t))} w\left( |T^*|^2 |T|^2\right) \\ && + \frac{6+ 5f(t)}{16(1+f(t))} w(T^2) \left\| |T|^2 + |T^*|^2 \right\|.
		\end{eqnarray*}
	\end{theorem}
	\begin{proof}
	Replacing $x$ with $Tx$, $y$ with $T^*x$, and $e$ with $x$ where $\|x\|=1$ in Lemma \ref{imp buzano}, we have 
		
	\begin{eqnarray*}
		|\langle Tx, x \rangle|^4 &\leq& \frac{2 + 3 f(t)}{8(1+ f(t))} \|Tx\|^2 \|T^*x\|^2 + \frac{6+ 5 f(t)}{8(1+ f(t))} \|Tx\| \|T^*x\| |\langle Tx, T^*x \rangle|\\
		&=& \frac{2 + 3 f(t)}{8(1+ f(t))} \langle |T|^2x, x \rangle \langle x, |T^*|^2x \rangle + \frac{6+ 5 f(t)}{8(1+ f(t))} \|Tx\| \|T^*x\| |\langle Tx, T^*x \rangle|\\
		&\leq& \frac{2 + 3 f(t)}{16(1+ f(t))} \left(\||T|^2x\| \cdot \||T^*|^2 x \| +\left| \left \langle |T|^2 x, |T^*|^2 x \right \rangle\right| \right)\\ &&+ \frac{6+ 5 f(t)}{16(1+ f(t))} \left( \|Tx\|^2 + \|T^*x\|^2 \right) |\langle T^2x, x \rangle|~~(\mbox{using Lemma \ref{buzano}})\\
		&\leq& \frac{2 + 3 f(t)}{32(1+ f(t))} \left(\||T|^2x\|^2 + \||T^*|^2x\| \right)+ \frac{2 + 3 f(t)}{16(1+ f(t))}\left| \left\langle |T^*|^2 |T|^2x, x \right\rangle \right| \\ && +  \frac{6+ 5 f(t)}{16(1+ f(t))} \left\langle\left\{ |T|^2 + |T^*|^2\right\}x,x \right\rangle |\langle T^2x, x \rangle|\\ &=&  \frac{2 + 3 f(t)}{32(1+ f(t))} \left\langle\left\{ |T|^4 + |T^*|^4\right\}x,x \right\rangle + \frac{2 + 3 f(t)}{16(1+ f(t))} \left|\left\langle |T^*|^2 |T|^2x, x \right\rangle \right|\\&&+  \frac{6+ 5 f(t)}{16(1+ f(t))} \left\langle\left\{ |T|^2 + |T^*|^2\right\}x,x \right\rangle |\langle T^2x, x \rangle|\\ &\leq& \frac{2 + 3 f(t)}{32(1+ f(t))} \left\| |T|^4 + |T^*|^4 \right\| + \frac{2 + 3 f(t)}{16 (1 + f(t))} w\left( |T^*|^2 |T|^2\right) \\ && + \frac{6+ 5f(t)}{16(1+f(t))} w(T^2) \left\| |T|^2 + |T^*|^2 \right\|.
	\end{eqnarray*}
Now, by taking the supremum over all $x$ with $\|x\|=1$, we obtain the desired inequality.
	\end{proof}
The following remark demonstrates that our inequality derived in Theorem \ref{th 4} represents a significant improvement over the inequality (\ref{eq el kit}) for $r=2$.
	\begin{remark}
			Let $T \in \mathcal{B}(\mathcal{H})$ and $f: (0, 1) \rightarrow \mathbb{R}^+$ be a well-defined function. Then 
		\begin{eqnarray*}
			w^4(T) &\leq& \frac{2 + 3 f(t)}{32(1+ f(t))} \left\| |T|^4 + |T^*|^4 \right\| + \frac{2 + 3 f(t)}{16 (1 + f(t))} w\left( |T^*|^2 |T|^2\right) \\ && + \frac{6+ 5f(t)}{16(1+f(t))} w(T^2) \left\| |T|^2 + |T^*|^2 \right\|\\
			&\leq& \frac{1}{2}  \left\| |T|^4 + |T^*|^4 \right\|.
		\end{eqnarray*}
	\end{remark}
	\begin{proof}
	The second inequality ensues from the application of inequality (\ref{eq dragomir}), combined with Lemmas (\ref{positive op}) and (\ref{convex op}), alongside the observation that $w(T^2) \leq w^2(T)$.
	\end{proof}
	Our subsequent lemma represents another improvement  of the Cauchy-Schwarz  inequality.
	\begin{lemma} \label{imp buzano 2}
		Let $x, y, e \in \mathcal{H},$ with $\|e\|=1.$ Then for a well-defined function $f:(0,1) \rightarrow \mathbb{R}^+$ 
		\begin{eqnarray*}
		\left|\langle x, e \rangle \langle e, y \rangle\right|^2 &\leq&  \frac{f(t)}{4(1+f(t))} \left(\|x\|^2 \|y\|^2 ++ |\langle x, y \rangle|^2 + 2 \|x\|\|y\| \langle x, y \rangle| \right)\\ && + \frac{1}{2(1+f(t))} \left|\langle x, e \rangle \langle e, y \rangle\right| \left( \|x\|\|y\| + |\langle x, y\rangle|\right).
		\end{eqnarray*}
	\end{lemma}
	\begin{proof}
	Utilizing Lemma \ref{buzano}, we get
		\begin{eqnarray*}
			\left|\langle x, e \rangle \langle e, y \rangle\right|^2 &=& \frac{f(t)}{(1+f(t))} \left|\langle x, e \rangle \langle e, y \rangle\right|^2 + \frac{1}{(1+f(t))} \left|\langle x, e \rangle \langle e, y \rangle\right|^2\\
			&\leq& \frac{f(t)}{4(1+f(t))} \left( \|x\|\|y\| + |\langle x, y\rangle|\right)^2 \\&&+ \frac{1}{2(1+f(t))} 	\left|\langle x, e \rangle \langle e, y \rangle\right| \left( \|x\|\|y\| + |\langle x, y\rangle|\right)\\
			&=&\frac{f(t)}{4(1+f(t))} \left(\|x\|^2 \|y\|^2 ++ |\langle x, y \rangle|^2 + 2 \|x\|\|y\| \langle x, y \rangle| \right)\\ && + \frac{1}{2(1+f(t))} \left|\langle x, e \rangle \langle e, y \rangle\right| \left( \|x\|\|y\| + |\langle x, y\rangle|\right).
		\end{eqnarray*}
	\end{proof}
	
Building upon the above lemma, we proceed to establish our next theorem.
\begin{theorem} \label{th5}
	Let $T \in \mathcal{B}(\mathcal{H}),$ and $f:(0,1) \rightarrow \mathbb{R}^+$ be a well-defined function. Then 
	\begin{eqnarray*}
		w^4(T) &\leq& \frac{f(t)}{16(1+f(t))} \left\||T|^4 + |T^*|^4
\right\| + \frac{f(t)}{8(1+f(t))} w\left(|T^*|^2 |T|^2 \right)\\ && + \frac{f(t)}{4(1+f(t))} w^2(T^2) + \frac{f(t)}{4(1+f(t))} 	\left\||T|^2 + |T^*|^2 \right\| w(T^2) \\ && + \frac{1}{4(1+f(t))} w^2(T)	\left\||T|^2 + |T^*|^2 \right\|  + \frac{1}{2(1+f(t))} w^2(T) w(T^2).
\end{eqnarray*}
\end{theorem}	
	
	\begin{proof}
		Let $x \in \mathcal{H},$with $\|x\|=1.$	Replacing $x$ with $Tx$, $y$ with $T^*x$, and $e$ with $x$ where $\|x\|=1$ in Lemma \ref{imp buzano 2}, we have 
		\begin{eqnarray*}
			|\langle Tx, x \rangle|^4 &\leq&  \frac{f(t)}{4(1+f(t))} \left(\|Tx\|^2 \|T^*x\|^2 + |\langle Tx, T^*x \rangle|^2 + 2 \|Tx\|\|T^*x\| \left|\langle Tx, T^*x \rangle| \right| \right)\\ && + \frac{1}{2(1+f(t))} \left|\langle Tx, x \rangle \langle x, T^*x \rangle\right| \left( \|Tx\|\|T^*x\| + |\langle Tx, T^*x\rangle|\right)
			\end{eqnarray*}
			\begin{eqnarray*}
			&\leq&  \frac{f(t)}{4(1+f(t))} \left\langle |T|^2x, x \right\rangle \left \langle x, |T^*|^2x \right\rangle +  \frac{f(t)}{4(1+f(t))} \left|\langle T^2x, x \rangle \right|^2\\ &&+ \frac{f(t)}{4(1+f(t))} \left(\|Tx\|^2+ \|T^*x\|^2 \right) |\langle T^2x, x \rangle|\\ && + \frac{1}{4(1+f(t))} |\langle Tx, x \rangle|^2 \left(\|Tx\|^2 + \|T^*x\|^2 \right)\\ && + \frac{1}{2(1+ f(t))}  |\langle Tx, x \rangle|^2 |\langle T^2x, x \rangle| \\
			&\leq& \frac{f(t)}{8(1+f(t))} \left(\||T|^2x \| \cdot\||T^*|^2x \| + \left|\left\langle |T|^2x, |T^*|^2x \right\rangle\right| \right) + \frac{f(t)}{4(1+f(t))} \left|\langle T^2x, x \rangle\right|^2 \\ && + \frac{f(t)}{4(1+f(t))}\left|\langle T^2x, x \rangle\right| \left\| |T|^2 + |T^*|^2 \right\|+ \frac{1}{4(1+f(t))} |\langle Tx, x \rangle|^2 \left\| |T|^2 + |T^*|^2 \right\| \\ && + \frac{1}{2(1+f(t))}|\langle Tx, x \rangle|^2 |\langle T^2x, x \rangle|~~(\mbox{using Lemma \ref{buzano}})\\
			&\leq& \frac{f(t)}{16(1+f(t))}\left \||T|^4 + |T^*|^4 \right \| + \frac{f(t)}{8(1+f(t))} w \left(|T^*|^2 |T|^2 \right) +\\ && \frac{f(t)}{4(1+f(t))} w^2 (T^2) + \frac{f(t)}{4(1+f(t))}w(T^2) \left\| |T|^2 + |T^*|^2 \right\|\\&&+ \frac{1}{4(1+f(t))} w^2(T) \left\| |T|^2 + |T^*|^2 \right\|  + \frac{1}{2(1+f(t))} w^2(T) w(T^2)
			\end{eqnarray*}
			Now, taking supremum over all $x$, with $\|x\| =1,$ we get our required inequality.
				\end{proof}
The following corollary can be obtained from the above-mentioned theorem.
	\begin{cor}
		Let $T \in \mathcal{B}(\mathcal{H}),$ and $f:(0,1) \rightarrow \mathbb{R}^+$ be a well-defined function. Then 
			\begin{eqnarray*}
			w^4(T) &\leq& \frac{f(t)}{16(1+f(t))} \left\||T|^4 + |T^*|^4
			\right\| + \frac{f(t)}{8(1+f(t))} w\left(|T^*|^2 |T|^2 \right)\\ && + \frac{f(t)}{4(1+f(t))} w^2(T^2) + \frac{f(t)}{4(1+f(t))} 	\left\||T|^2 + |T^*|^2 \right\| w(T^2) \\ && + \frac{1}{4(1+f(t))} w^2(T)	\left\||T|^2 + |T^*|^2 \right\|  + \frac{1}{2(1+f(t))} w^2(T) w(T^2)\\
			&\leq& \frac{1}{2}  \left\||T|^4 + |T^*|^4
			\right\| .
		\end{eqnarray*}
	\end{cor}
		\begin{proof}
	The second inequality follows from inequality (\ref{eq dragomir}), Lemma (\ref{positive op}), (\ref{convex op}), and the numerical radius power inequality. Therefore, our inequality derived in Theorem \ref{th5} improves the inequality (\ref{eq el kit}) for $r=2$.
	\end{proof}
	If we select $f(t) = \frac{1}{2}$ in Theorem \ref{th5}, we obtain the subsequent inequality, which improves the result presented in \cite[Th. 3]{Kittaneh MIA}.
	\begin{remark}
		Let $T \in \mathcal{B}(\mathcal{H}).$ Then 
		\begin{eqnarray*}
			w^4(T) &\leq& \frac{1}{48}  \left\||T|^4 + |T^*|^4
			\right\|  + \frac{1}{24} w\left(|T^*|^2 |T|^2 \right) \\&& + \frac{1}{12} \left(w^2(T^2) + w(T^2)  \left\||T|^2 + |T^*|^2
			\right\|  \right) \\ && + \frac{1}{3} \left(\frac{1}{2}  \left\||T|^2 + |T^*|^2	\right\| + w(T^2) \right).
		\end{eqnarray*}
	\end{remark}
	
	\begin{lemma} \label{imp buzano 3}
		Let $x, y, e \in \mathcal{H}$ with $\|e\|=1,$ and $n \geq 1.$ Then for any well-defined function $f:(0,1) \rightarrow \mathbb{R}^+,$ 
		\begin{eqnarray*}
			\left| \langle x, e \rangle \langle e, y \rangle\right|^{2n} &\leq& \frac{1}{2^{2n}} \frac{1+ 2 f(t)}{1+ f(t)} \|x\|^{2n} \|y\|^{2n} + \frac{1}{2^{2n} (1+f(t))} \|x\|^n \|y\|^n |\langle x, y \rangle|^n \\ && + \frac{1}{2^{2n}} \sum_{r=1}^{2n-1} \binom{2n}{r} \|x\|^r \|y\|^r |\langle x, y \rangle|^{2n-r}.
		\end{eqnarray*}
	\end{lemma}
	\begin{proof}
		Using Lemma \ref{buzano} we have,
		\begin{eqnarray*}
				\left| \langle x, e \rangle \langle e, y \rangle\right|^{2n} &\leq& \frac{1}{2^{2n}} \left(\|x\|\|y\| + |\langle x, y \rangle | \right)^{2n}\\
				&=& \frac{1}{2^{2n}} \left(\|x\|^{2n} \|y\|^{2n} +  |\langle x, y \rangle|^{2n} +    \sum_{r=1}^{2n-1} \binom{2n}{r} \|x\|^r \|y\|^r |\langle x, y \rangle|^{2n-r} \right)\\ &&~~\,\,\,\,\,\,\,\,\,\,\,\,\,\,\,\,\,\,\,\,\,\,\,\,\,\,\,\,\,\,\,\,\,\,\,\,\,\,\,\,\,\,\,\,\,\,\,\,\,\,\,\,\,\,\,\,\,\,\,\,\,\,\,\,\,\,\,\,\,\,\,\,\,\,\,\,\,\,\,\,\,~(\mbox{using Binomial theorem})\\
				&\leq& \frac{1}{2^{2n}} \|x\|^{2n} \|y\|^{2n} + \frac{1}{2^{2n}} \left(\frac{f(t)}{1+ f(t)} \|x\|^2 \|y\|^2 + \frac{1}{1+ f(t)} \|x\| \|y\| |\langle x, y \rangle| \right)^n \\ && + \frac{1}{2^{2n}} \sum_{r=1}^{2n-1} \binom{2n}{r} \|x\|^r \|y\|^r |\langle x, y \rangle|^{2n-r}~~(\mbox{using Lemma \ref{gen cauchy}})\\
				&=& \frac{1}{2^{2n}} \frac{1+ 2 f(t)}{1+ f(t)} \|x\|^{2n} \|y\|^{2n} + \frac{1}{2^{2n} (1+f(t))} \|x\|^n \|y\|^n |\langle x, y \rangle|^n \\ && + \frac{1}{2^{2n}} \sum_{r=1}^{2n-1} \binom{2n}{r} \|x\|^r \|y\|^r |\langle x, y \rangle|^{2n-r}.
		\end{eqnarray*}
	\end{proof}
	Using the above lemma we will establish our next theorem.
	\begin{theorem} \label{th6}
		Let $T \in \mathcal{B}(\mathcal{H}),$ and $f: (0,1) \rightarrow \mathbb{R}^+$ be a well-defined function. Then for any $n \geq 1,$ 
		\begin{eqnarray*}
			w^{4n}(T)  &\leq & \frac{1}{2^{2n +2}} \frac{1+ 2 f(t)}{1+ f(t)} \left\| |T|^{4n} + |T^*|^{4n} \right\| +\frac{1}{2^{2n +1}} \frac{1+ 2 f(t)}{1+ f(t)} w \left(|T^*|^{2n} |T|^{2n} \right)  \\ &&+ \frac{1}{2^{2n +1}} \frac{1}{1+ f(t)} \left\| |T|^{2n} + |T^*|^{2n} \right\| w^n(T^2) \\&&+ \frac{1}{2^{2n+1}} \sum_{r=1}^{2n-1} \binom{2n}{r} \left\| |T|^{2r} + |T^*|^{2r} \right\| w^{2n-r} (T^2).
		\end{eqnarray*}
	\end{theorem}
	\begin{proof}
	Replacing $x$ with $Tx$, $y$ with $T^*x$, and $e$ with $x$ with $\|x\|=1$ in Lemma \ref{imp buzano 3}, we obtain
		\begin{eqnarray*}
			\left| \langle Tx, x \rangle \right|^{4n} &\leq& \frac{1}{2^{2n}} \frac{1+ 2 f(t)}{1+ f(t)} \|Tx\|^{2n} \|T^*x\|^{2n} + \frac{1}{2^{2n} (1+f(t))} \|Tx\|^n \|T^*x\|^n |\langle Tx, T^*x \rangle|^n \\ && + \frac{1}{2^{2n}} \sum_{r=1}^{2n-1} \binom{2n}{r} \|Tx\|^r \|T^*x\|^r |\langle Tx, T^*x \rangle|^{2n-r}\\ &\leq& \frac{1}{2^{2n}} \frac{1+ 2 f(t)}{1+ f(t)} \left\langle |T|^{2n}x, x \right\rangle \left\langle x, |T^*|^2 x \right\rangle \\ &&+ \frac{1}{2^{2n+1} (1+f(t))} \left\langle \left(|T|^{2n} +  |T^*|^{2n} \right)x, x \right \rangle |\langle T^2x, x \rangle|^n \\ && + \frac{1}{2^{2n+1}} \sum_{r=1}^{2n-1} \binom{2n}{r} \left\langle \left( |T|^{2r} + |T^*|^{2r} \right) x, x \right \rangle |\langle T^2 x , x \rangle|^{2n - r}\\
			&\leq&  \frac{1}{2^{2n +1}} \frac{1+ 2 f(t)}{1+ f(t)} \left(\| |T|^{2n}x\| \||T^*|^{2n}x \| + \left|\left\langle |T^*|^{2n} |T|^{2n} x, x \right\rangle  \right| \right)\\ && +  \frac{1}{2^{2n+1} (1+f(t))} \left\langle \left\{|T|^{2n} +  |T^*|^{2n} \right\}x, x \right \rangle |\langle T^2x, x \rangle|^n \\ && + \frac{1}{2^{2n+1}} \sum_{r=1}^{2n-1} \binom{2n}{r} \left\langle \left( |T|^{2r} + |T^*|^{2r} \right) x, x \right \rangle |\langle T^2 x , x \rangle|^{2n - r}~~(\mbox{using Lemma \ref{buzano}})\\
			&\leq& \frac{1}{2^{2n +2}} \frac{1+ 2 f(t)}{1+ f(t)} \left\| |T|^{4n} + |T^*|^{4n} \right\| +\frac{1}{2^{2n +1}} \frac{1+ 2 f(t)}{1+ f(t)} w \left(|T^*|^{2n} |T|^{2n} \right)  \\ &&+ \frac{1}{2^{2n +1}} \frac{1}{1+ f(t)} \left\| |T|^{2n} + |T^*|^{2n} \right\| w^n(T^2) \\&&+ \frac{1}{2^{2n+1}} \sum_{r=1}^{2n-1} \binom{2n}{r}\left\| |T|^{2r} + |T^*|^{2r} \right\| w^{2n-r} (T^2).
		\end{eqnarray*}
	Now, by taking the supremum over all $x$ with $\|x\| =1$, we obtain our required inequality.
	\end{proof}

	If we substitute $n=1$ and apply inequality (\ref{eq dragomir}) and Lemma \ref{convex op} in Theorem \ref{th6}, we derive the following corollary.
	\begin{cor} \label{imp bomi}
			Let $T \in \mathcal{B}(\mathcal{H}),$ and $f: (0,1) \rightarrow \mathbb{R}^+$ be a well-defined function. Then  
			\begin{eqnarray*}
				w^4(T) &\leq&  \frac{1+ 2 f(t)}{8(1+ f(t))} \left\| |T|^4 + |T^*|^4 \right\| +  \frac{3+ 2 f(t)}{8(1+ f(t))} \left\| |T|^2 + |T^*|^2 \right\| w(T^2)\\
				&\leq& \frac{1}{2}  \left\| |T|^4 + |T^*|^4 \right\|.
			\end{eqnarray*}
	\end{cor}
	\begin{remark}
	It is evident that the inequality derived in Corollary \ref{imp bomi} extends the two upper bounds established by Bomi-Domi et al. \cite[Th. 2.1]{Bomi} and Omidvar et al. \cite[Th. 2.1]{omidvar} for $f(t)=1$.
	For $T \in \mathcal{B}(\mathcal{H}),$ 
		 \begin{eqnarray*}
			w^4(T) &\leq&  \frac{3}{16}  \left\| |T|^4 + |T^*|^4 \right\| + \frac{5}{16}  \left\| |T|^2 + |T^*|^2 \right\| w(T^2)\\ &\leq& \frac{3}{8}  \left\| |T|^4 + |T^*|^4 \right\| + \frac{1}{8}  \left\| |T|^2 + |T^*|^2 \right\| w(T^2).
			\end{eqnarray*}
	\end{remark}

\noindent \textbf{Declarations:} \\

\textbf{Conflict of Interest:} The authors declare that there is no conflict of interest.\\

\textbf{Funding:} This research was conducted without any specific funding from external agencies.	\\	

\textbf{Data Availability Statement:}	No new data were created or analyzed in this study. Data sharing is not applicable to this article.

			\bibliographystyle{amsplain}
			
\end{document}